\documentclass[11pt]{article}\usepackage{amsfonts,amssymb,amsmath,latexsym,xcolor,epsfig}
\usepackage[textsize=scriptsize,colorinlistoftodos]{todonotes}
\usepackage[mathscr]{euscript}
\usepackage{hyperref}
\usepackage{amsthm}
\let\oldcite=\cite 
\renewcommand\cite[1]{\ifthenelse{\equal{#1}{need}}{\todo{citation needed}}{\oldcite{#1}}}   

\newtheorem{theorem}{Theorem}

\newtheorem{lemma}[theorem]{Lemma}

\newtheorem{claim}{Claim}

\def\qed{\ifvmode\mbox{ }\else\unskip\fi\hskip 1em plus 10fill$\Box$}
\input{epsf}

\makeatletter
\def\Ddots{\mathinner{\mkern1mu\raise\p@
\vbox{\kern7\p@\hbox{.}}\mkern2mu
\raise4\p@\hbox{.}\mkern2mu\raise7\p@\hbox{.}\mkern1mu}}
\makeatother

\author{
Jacob Fox\thanks{Department of Mathematics, Stanford University, Stanford, CA 94305. Email: {\tt jacobfox@stanford.edu}. Research supported by a Packard Fellowship, by NSF Career Award DMS-1352121 and by an Alfred P. Sloan Fellowship.}
\and
L\'aszl\'o Mikl\'os Lov\'asz\thanks{Department of Mathematics,  Massachusetts Institute of Technology,  Cambridge,  MA 02139-4307.   Email:
{\tt lmlovasz@math.mit.edu}.}
}
\title{\vspace{-0.7cm} {A tight bound for Green's arithmetic triangle removal lemma in vector spaces}}

\begin{document}
\maketitle 
\begin{abstract}
Let $p$ be a fixed prime. A triangle in $\mathbb{F}_p^n$ is an ordered triple $(x,y,z)$ of points satisfying $x+y+z=0$. Let $N=p^n=|\mathbb{F}_p^n|$. Green proved an arithmetic triangle removal lemma which says that for every $\epsilon>0$ and prime $p$, there is a $\delta>0$ such that if $X,Y,Z \subset \mathbb{F}_p^n$ and the number of triangles in $X \times Y \times Z$ is at most $\delta N^2$, then we can delete $\epsilon N$ elements from $X$, $Y$, and $Z$ and remove all triangles. Green posed the problem of improving the quantitative bounds on the arithmetic triangle removal lemma, and, in particular, asked whether a polynomial bound holds. Despite considerable attention, prior to this paper, the best known bound, due to the first author, showed that $1/\delta$ can be taken to be an exponential tower of twos of height logarithmic in $1/\epsilon$. 

We solve Green's problem, proving an essentially tight bound for Green's arithmetic triangle removal lemma in $\mathbb{F}_p^n$. We show that a polynomial bound holds, and further determine the best possible exponent. Namely, there is a explicit number $C_p$ such that we may take $\delta = (\epsilon/3)^{C_p}$, and we must have $\delta \leq \epsilon^{C_p-o(1)}$. In particular,  $C_2=1+1/(5/3 - \log_2 3) \approx 13.239$, and $C_3=1+1/c_3$ with $c_3=1-\frac{\log b}{\log 3}$, $b=a^{-2/3}+a^{1/3}+a^{4/3}$, and  $a=\frac{\sqrt{33}-1}{8}$, which gives $C_3 \approx 13.901$. The proof uses the essentially sharp bound on multicolored sum-free sets due to work of Kleinberg-Sawin-Speyer, Norin, and Pebody,  which builds on the recent breakthrough on the cap set problem by Croot-Lev-Pach, and the subsequent work by Ellenberg-Gijswijt, Blasiak-Church-Cohn-Grochow-Naslund-Sawin-Umans, and Alon.


\end{abstract}

\section{Introduction}

Let $p$ be a fixed prime.  A {\it triangle} in $\mathbb{F}_p^n$ is a triple $(x,y,z)$ of points with $x+y+z=0$. We use $N=|\mathbb{F}_p^n|=p^n$ throughout. The arithmetic triangle removal lemma for vector spaces over finite fields states that for each $\epsilon>0$ and prime $p$, there is a $\delta=\delta(\epsilon,p)>0$ such that if $X,Y,Z \subset \mathbb{F}_p^n$, and the number of triangles in $X \times Y \times Z$ is at most $\delta N^2$, then we can delete $\epsilon N$ points from $X$, $Y$, and $Z$ so that no triangle remains. It was originally proved by Green \cite{GREEN05} using a Fourier analytic arithmetic regularity lemma, with a bound on $1/\delta$ which is a tower of twos of height polynomial in $1/\epsilon$. Kr\'al', Serra, and Vena \cite{KSV09} showed that the arithmetic triangle removal lemma follows from the triangle removal lemma for graphs. Green's proof shows the result further holds in any abelian group, and the Kr\'al'-Serra-Vena proof shows that an analogue holds in any group. 

Green \cite{GREEN05} posed the problem of improving the quantitative bounds on the arithmetic triangle removal lemma, and, in particular, asked whether a polynomial bound holds. Green's problem has received considerable attention by many researchers \cite{BHATTACHARYYA13,BCSX10,BGRS12,BGS15,BX09,FOX11,FK14,GREEN05a,
GREEN05,HST16,HX15,RS11}. This is in part due to its applications and connections to several major open problems in number theory, combinatorics, and computer science. The first author \cite{FOX11} gave an improved bound on the triangle removal lemma for graphs. Together with the Kr\'al'-Serra-Vena reduction, it gives a bound on $1/\delta$ in the arithmetic triangle removal lemma which is a tower of twos of height logarithmic in $1/\epsilon$, see \cite{FOX11} for details. An alternative, Fourier-analytic proof of this bound in the case when $p=2$ was given by Hatami, Sachdeva, and Tulsiani \cite{HST16}. Prior to this paper, this tower-type bound was the best known bound for the arithmetic triangle removal lemma. 

In the case $p=2$, obtaining better upper bounds for $\delta$ has also received much attention due in part to its close connection to property testing, and, in particular, testing triangle-freeness. Bhattacharyya and Xie  \cite{BX09} were the first to give a nontrivial upper bound, showing that we must have $\delta \leq \epsilon^{4.847}$. Fu and Kleinberg \cite{FK14} provided a simple construction showing that we must have $\delta \leq \epsilon^{C_2-o(1)}$, where $C_2 \approx 13.239$ is defined below. It is based on a construction from Coppersmith and Winograd's famous matrix multiplication algorithm \cite{CW90}. It was widely conjectured that the exponent in the bound is not optimal (and that the bound is perhaps even superpolynomial), and Haviv and Xie \cite{HX15} introduced an approach to obtaining better upper bounds on the arithmetic triangle removal lemma.

In this paper, we solve Green's problem by proving an essentially tight bound for the arithmetic triangle removal lemma in vector spaces over finite fields. We show that a polynomial bound holds, and further determine the best possible exponent for every prime $p$.
It involves, for every $p$, constants $c_p$ which will be between $0$ and $1$ and will be defined below, and $C_p=1+1/c_p$. In particular, we have\footnote{Here, and throughout the paper, all logarithms are base $2$.} $c_2=5/3-\log 3 \approx .0817$ and $C_2=1+1/c_2 \approx 13.239$ as in Fu and Kleinberg \cite{FK14}, and $c_3 = 1-\frac{\log b}{\log 3}$, where $b=a^{-2/3}+a^{1/3}+a^{4/3}$, and $a=\frac{\sqrt{33}-1}{8}$. This gives $c_3 \approx .0775$, and $C_3 \approx 13.901$. 

\begin{theorem} \label{main}
Let $0< \epsilon < 1$ and $\delta =  (\epsilon/3)^{C_p}$. If $X,Y,Z \subset \mathbb{F}_p^n$  with less than $\delta N^2$ triangles in $X \times Y \times Z$, then we can remove $\epsilon N$ elements from $X \cup Y \cup Z$ so that no triangle remains. Furthermore, this bound is essentially tight in that the minimum possible $\delta=\delta(\epsilon,p)$ in the arithmetic triangle removal lemma satisfies  $\delta \leq \epsilon^{C_p-o(1)}$, where the $o(1)$ term goes to $0$ for $p$ fixed and $\epsilon \to 0$. 
\end{theorem}


Following a breakthrough by Croot, Lev, and Pach \cite{CLP16} on $3$-term arithmetic progressions in $\mathbb{Z}_4^n$, Ellenberg and Gijswijt \cite{EG16} showed that their method can be used to greatly improve the upper bound on the cap set problem, and more generally to bound the size of a set in $\mathbb{F}_p^n$ with no $3$-term arithmetic progressions. 
Further work by Blasiak-Church-Cohn-Grochow-Naslund-Sawin-Umans \cite{BCCGNSU16}, and independently Alon, established that their proof also shows an upper bound on the following multicolored sum-free theorem over $\mathbb{F}_p^n$. 
Later work by Kleinberg, Sawin, and Speyer \cite{KSS16}, conditional on a conjecture that was later proven by Norin \cite{NORIN16}, and Pebody \cite{PEBODY16} 
shows that the exponent $c_p$ obtained for the upper bound is sharp. The exponent is given by
\[p^{1-c_p} = \inf_{0<x<1} x^{-(p-1)/3}\left(x^0+x^1+\cdots+x^{p-1}\right)
.\]
It is not difficult to check that $c_p=\Theta((\log p)^{-1})$ and hence $C_p=\Theta(\log p)$. 




\begin{theorem} \label{sumsetresult} \cite{KSS16}
Given a collection of ordered triples $\{(x_i,y_i,z_i)\}_{i=1}^m$ in $\mathbb{F}_p^n$ such that $x_i+y_j+z_k=0$ holds if and only if $i=j=k$, the size of the collection satisfies the bound
\[m \le p^{(1-c_p)n}.\]
Furthermore, there exists such a collection with $m \ge p^{(1-c_p)n-o(n)}$.
\end{theorem}
The lower bound in the case of $p=2$ was established earlier by Fu and Kleinberg \cite{FK14}.

We note that while our proof of Theorem \ref{main}, determining the exponential constant, relies on Theorem \ref{sumsetresult}, the existence of the exponential constant follows from the existence of a polynomial bound (which is implied by the earlier work of Ellenberg and Gijswijt \cite{EG16}), and a simple product trick, discussed in Section \ref{lowerbound}.

We give our proof of the lower bound in the next section. In the following section, for completeness, we give the argument that shows the matching upper bound. We finish with some concluding remarks.

\section{Proof of the lower bound} \label{lowerbound}
Recall that $N=p^n$. We prove the following theorem, which is roughly equivalent to Theorem \ref{main}.
\begin{theorem} \label{main1}
Suppose we have $m=\epsilon N$ disjoint triangles $\{(x_i,y_i,z_i)\}_{i=1}^m$ in $\mathbb{F}_p^n$. Let $\delta=\epsilon^{C_p}$. Then, for $X=\{x_i\}_{i=1}^m$, $Y=\{y_i\}_{i=1}^m$, and $Z=\{z_i\}_{i=1}^m$, we have at least $\delta N^2$ triangles in $X \times Y \times Z$.
\end{theorem}

\begin{proof}[Proof that Theorem \ref{main1} implies Theorem \ref{main}] Suppose we have $X$, $Y$, and $Z$ such that we cannot remove $\epsilon N$ points to remove all triangles. In this case, we can find a set of at least $\frac{\epsilon}{3}N$ disjoint triangles $(x_i,y_i,z_i)$. Indeed, start taking out disjoint triangles greedily from $X \times Y \times Z$. If at some point we cannot take any more and we have taken less than $\frac{\epsilon}{3}N$ triangles, then remove any point in any of these triangles. We would have removed less than $\epsilon N$ points and we would have no triangles left, a contradiction. So we can find $\frac{\epsilon}{3}N$ disjoint triangles, and so we can apply Theorem \ref{main1} with $\epsilon/3$ to find $\delta N^2$ triangles, a contradiction.
\end{proof}

To obtain Theorem \ref{main1}, we first prove the following weaker, asymptotic version. We will later use a simple product amplification trick to remove the $o(1)$ in the exponent and obtain Theorem \ref{main1}. 
\begin{theorem} \label{mainasympt}
Suppose we have $m=\epsilon N$ disjoint triangles $\{(x_i,y_i,z_i)\}_{i=1}^m$ in $\mathbb{F}_p^n$. Then, for $X=\{x_i\}_{i=1}^m$, $Y=\{y_i\}_{i=1}^m$, and $Z=\{z_i\}_{i=1}^m$, we have at least $\delta N^2$ triangles in $X \times Y \times Z$, where $\delta=\epsilon^{C_p+o(1)}$. 
\end{theorem}

We first prove two lemmas. The following lemma bounds the total number of triangles based on the maximal degree of a point in the union of the sets. In particular, this proves the theorem in the case when the maximum degree is not much larger than the average degree.

\begin{lemma} \label{roughlyeven}
Suppose $0<\rho \le 1/(5p^3)$ and we have disjoint sets $X,Y,Z \subset \mathbb{F}_p^n \setminus \{0\}$, such that any two vectors from their union are linearly independent, any two-dimensional subspace contains at most one triangle in $X \times Y \times Z$, and each element in $X \cup Y \cup Z$ is in at most $\rho N$ triangles in $X \times Y \times Z$. Let $\delta N^2$ be the total number of triangles in $X \times Y \times Z$. Then $\delta \le 125p^2 \rho^{1+c_p}$.
\end{lemma}

\begin{proof}
Let $d=\lfloor \frac{\log (1/(5\rho))}{\log p} \rfloor$, so $1/(5p) < \rho p^d \le 1/5$. 
Note that $d \ge 3$. Take a uniformly random subspace $U$ of $\mathbb{F}_p^n$ of dimension $d$. Let $X'=X \cap U$, and define $Y'$ and $Z'$ analogously. We call a triangle $(x,y,z) \in X' \times Y' \times Z'$ \emph{good} if it is the unique triangle in $X' \times Y' \times Z'$ containing $x$, $y$, or $z$.

\begin{claim}
Given a triangle $(x,y,z) \in X \times Y \times Z$, conditioned on $x \in X',y \in Y',z \in Z'$, the probability that it is good is at least $2/5$.
\end{claim}

\begin{proof}
We know that $x$ is in at most $\rho N$ triangles in $X \times Y \times Z$. These are each formed with a different element of $Y$. Let $Y_1$ be the set of elements in $Y$ that together with $x$ are in a triangle in $X\times Y \times Z$. By assumption, they each generate a two-dimensional subspace with $x$, and different elements of $Y$ generate different two-dimensional subspaces. For any element in $Y_1$ not equal to $y$, the probability that it is in $U$, conditioned on $x,y,z \in U$, is
\[\frac{p^{d-2}-1}{p^{n-2}-1}.
\]
Thus, using the union bound, the probability that any other element of $Y_1$ is in $U$ is at most
\[\rho N \frac{p^{d-2}-1}{p^{n-2}-1}=\rho p^n\frac{p^{d-2}-1}{p^{n-2}-1} \le \rho p^n \frac{p^{d-2}}{p^{n-2}}= \rho p^d \le \frac{1}{5}.
\]
So, conditioned on $x \in X'$, $y \in Y'$, $z \in Z'$, the probability that either $x$, $y$, or $z$ is in another triangle in $X' \times Y' \times Z'$ that is not in the two-dimensional subspace generated by $\{x,y,z\}$ is at most $\frac{3}{5}$.
\end{proof}

%

\begin{claim}
Let $T$ be the number of good triangles. Then $\mathbb{E}[T] \ge\frac{\delta}{125p^2\rho^2}$.
\end{claim}

\begin{proof}
For each triangle in $X \times Y \times Z$, it has a probability of $\frac{(p^d-1)(p^{d-1}-1)}{(p^n-1)(p^{n-1}-1)}$ of being in $X' \times Y' \times Z'$, and conditioned on this, it has a probability of at least $2/5$ of being good. Since there are $\delta N^2$ triangles in total,
\[\mathbb{E}[T] \ge \delta N^2\frac{2}{5}\frac{(p^d-1)(p^{d-1}-1)}{(p^n-1)(p^{n-1}-1)} \ge \frac{1}{5}\delta p^{2n} \frac{p^{2d}}{p^{2n}} \ge \frac{\delta}{125p^2\rho^2}
\]
\end{proof}

Observe that the set of good triangles satisfies the hypothesis of Theorem \ref{sumsetresult}.
%
Therefore, 
$T$ is always at most
\[p^{(1-c_p)d}.
\]
We thus have
\[p^{(1-c_p)d} \ge \mathbb{E}[T] \ge \frac{\delta}{125p^2\rho^2} ,
\] which, combined with the fact that $p^d \le \frac{1}{5\rho} \le \frac{1}{\rho}$, implies that 
\[ \delta \le 125p^2\rho^2 (p^d)^{1-c_p} \le 125p^2 \rho^{1+c_p}.\]
This completes the proof of Lemma \ref{roughlyeven}.
\end{proof}

Let $0<a_p\le 5/p^4$ and $g:(0,a_p] \to \mathbb{R}^+$ be any function such that $g(\beta)$ increases as $\beta$ decreases, but $g(\beta)\beta$ and $\beta^{c_p}g(\beta)^{1+c_p}$ both decrease as $\beta$ decreases, and $\sum_{i = 1}^{\infty} \frac{1}{g(2^{-i}a_p)} \leq 1/4$. We will later take for convenience $g(\beta)=\log^2(1/\beta)$, which satisfies the conditions for some $a_p>0$.

\begin{lemma} \label{lemmadisjointtriples}
Let $0<\delta \le a_p$ and $\epsilon \ge  (125p^2)^{\frac{1}{1+c_{p}}} \delta^{\frac{c_p}{1+c_p}} g(\delta)$. Suppose $m=\epsilon N$ and we have a collection of pairwise disjoint triangles $(x_i,y_i,z_i)$ for $i \in [m]$, and let $X=\{x_1,x_2,\ldots,x_{m}\}$, $Y=\{y_1,y_2,\ldots,y_m\}$, and $Z=\{z_1,z_2,\ldots,z_m\}$. Suppose that any two vectors from $X \cup Y \cup Z$ are linearly independent, and any two-dimensional subspace contains at most one triangle. Then there must be at least $\delta N^2$ triangles in $X \times Y \times Z$.
\end{lemma}

\begin{proof}
Suppose for contradiction that we have less than $\delta N^2$ triangles. 
We would like to obtain large subsets $X_1 \subset X$, $Y_1 \subset Y$, $Z_1 \subset Z$ such that no element of $X_1$, $Y_1$, or $Z_1$ is contained in significantly more triangles in $X_1 \times Y_1 \times Z_1$ than the average. This will allow us to apply Lemma \ref{roughlyeven} to obtain 
a contradiction.

We do this by removing bad points from the sets one at a time, based on the density of triangles. 
Suppose that after removing a certain number of points, we are left with $\delta' N^2$ triangles. 
If there exists a point that is in at least $g(\delta') \frac{\delta'}{\epsilon}N$ triangles, we remove it. We then update $\delta'$ and repeat.
We claim that when this process ends, we only removed at most $m/2$ points. Indeed, for any $\beta$, while the number of triangles is between $2\beta N^2$ and $\beta N^2$, we remove at least $g(\beta)\frac{\beta}{\epsilon} N$ triangles in each step, but since throughout this period the number of triangles is at most $2\beta N^2$, this happens at most $2\beta N^2/\left(g(\beta)\frac{\beta}{\epsilon} N\right) =\frac{2\epsilon}{g(\beta)}N$ times. Overall, this implies that we end up removing at most
\[\frac{2\epsilon}{g(\delta/2)}N+\frac{2\epsilon}{g(\delta/4)}N+\cdots \le \frac{\epsilon}{2}N
\]
points. So after this process, we are left with sets $X',Y',Z'$ which contain at least $\frac{\epsilon}{2} N$ disjoint triangles between them, and for some $\delta'$, we have $\delta' N^2$ triangles in $X' \times Y' \times Z'$, and every point in each set is in at most $g(\delta')\frac{\delta'}{\epsilon} N$ such triangles.
%

We can now apply Lemma \ref{roughlyeven} with 
\[\rho=g(\delta')\frac{\delta'}{\epsilon} \le g(\delta) \frac{\delta}{\epsilon} \le \left( \frac{\delta}{125p^2} \right)^{\frac{1}{1+c_p}} \le \frac{1}{5p^3}
.\]
Here the first inequality follows from the fact that $g(\beta) \beta$ decreases as $\beta$ decreases, the second from the bound on $\epsilon$, and the third from the bound $\delta \le a_p \le 5/p^4$ and the fact that $0<c_p<1$.
The lemma implies that 
\[\delta' \le 125p^2 \left(g(\delta')\frac{\delta'}{\epsilon}\right)^{1+c_p}
\]
which means that 
\[\epsilon \le  (125p^2)^{\frac{1}{1+c_{p}}} \delta'^{\frac{c_p}{1+c_p}} g(\delta')
\le (125p^2)^{\frac{1}{1+c_{p}}} \delta^{\frac{c_p}{1+c_p}} g(\delta),
\]
where we used for the second inequality that if $\delta' < a_p$, then $\delta'^{c_p}g(\delta')^{1+c_p}$ increases as $\delta'$ increases.
\end{proof}

\begin{proof}[Proof of Theorem \ref{mainasympt}]
We are now ready to prove Theorem \ref{mainasympt}. First, we would like to argue that we may assume that the sets $X$, $Y$, and $Z$ are pairwise disjoint, any two vectors are linearly independent, and any two-dimensional subspace contains at most one triangle in $X \times Y \times Z$. Let us work in $\mathbb{F}_p^{n+2}$, so we add two coordinates to each vector. For each $x_i$, we define $x_i'=(x_i,1,0)$. We also take $y_i'=(y_i,-1,1)$, and $z_i'=(z_i,0,-1)$. Then, for each $i$, we still have $x_i'+y_i'+z_i'=0$, and if we take  $X'=\{x_i'\}_{i=1}^m$, $Y'=\{y_i'\}_{i=1}^m$, and $Z'=\{z_i'\}_{i=1}^m$, then the triangles in $X' \times Y' \times Z'$ correspond exactly to the triangles in $X \times Y \times Z$. It is easy to see that the three new sets are disjoint, simply because of their last two coordinates. The last two coordinates also imply that if we multiply a point from one of the sets by a scalar not equal to $1$, we cannot be in the union of the sets. Thus, any two vectors from $X \cup Y \cup Z$ are linearly independent. Finally, fix a triangle $(x_i,y_j,z_k)$, and let $U$ be the two-dimensional subspace generated by them. Suppose that for some $j'$ with $j \ne j'$, the subspace generated by $x_i$ and $y_j$ contains $y_{j'}$. This means that 
\[y_{j'} = \alpha x_i + \beta y_j.
\] Just by looking at the last two coordinates, we clearly must have $\alpha=0$ and $\beta=1$, but then $y_j=y_{j'}$, a contradiction. This means that any triangle contained in $U$ must contain $y_j$. It is easy to analogously show that any triangle in $U$ must contain $x_i$. This implies that the only triangle contained in $U$ is $(x_i,y_j,z_k)$. Therefore, any two-dimensional subspace contains at most one triangle. 

Since the size of the underlying space is now $N'=p^2N$, we obtain $\epsilon'N'$ triangles with $\epsilon'=\epsilon/p^2$.
Using Lemma \ref{lemmadisjointtriples}, if we take $\delta'=\delta/p^4$, we have that there must be at least $\delta' N'^2=\delta N^2$ triangles in total, provided that we have
\[\frac{\epsilon}{p^2} >  (125p^2)^{\frac{1}{1+c_p}} (\delta/p^4) ^{\frac{c_p}{1+c_p}} g(\delta/p^4).
\] 
With our choice of $g(\rho)=(\log(1/\rho))^2$, if we take 
$\delta = \Omega_p\Big(\epsilon^{1+1/c_p}(\log(1/\epsilon))^{-\frac{2+2c_p}{c_p}}\Big)$, the conditions of the lemma are satisfied if $\epsilon$ (and thus $\delta$) are small enough. This gives $\delta \ge \epsilon^{C_p+o(1)}$.
\end{proof}

We now tighten the asymptotic bound in Theorem \ref{mainasympt} to prove Theorem \ref{main1} with a product argument, inspired by a similar one Kleinberg, Sawin, and Speyer \cite{KSS16} used for the multicolor sum-free problem.
\begin{proof}[Proof of Theorem \ref{main1}]
Suppose we have $m=\epsilon N$ disjoint triangles $\{(x_i,y_i,z_i)\}_{i=1}^m$ in $(\mathbb{F}_p^n)^3$. Define $X=\{x_i\}_{i=1}^m$, $Y=\{y_i\}_{i=1}^m$, $Z=\{z_i\}_{i=1}^m$, and suppose for the sake of contradiction there are only  $\delta N^2$ triangles in $X \times Y \times Z$ with $\delta <\epsilon^{C_p}$, so we can write $\delta=\epsilon^{(1+\alpha)C_p}$ for some $\alpha>0$. We define for every positive integer $k$ a collection of $m^k$ disjoint triangles in $(\mathbb{F}_p^{nk})^3$. For $k$ vectors $x_1,x_2,...,x_k$ in $\mathbb{F}_p^n$, we can take their concatenation $(x_1,x_2,...,x_k) \in \mathbb{F}_p^{nk}$. Then, we can take for any $k$-tuple $i_1,i_2,...,i_k$ the triangle $\left((x_{i_1},x_{i_2},...,x_{i_k}),(y_{i_1},y_{i_2},...,y_{i_k}),(z_{i_1},z_{i_2},...,z_{i_k})\right)$. It is easy to see that these are indeed triangles and disjoint, and if we define $X_k$, $Y_k$, and $Z_k$ as subsets of $\mathbb{F}_p^{nk}$ analogously to $X$, $Y$, and $Z$, then $X_k=X^k$, $Y_k=Y^k$, and $Z_k=Z^k$. Then any triangle must be a triangle in the first $n$ coordinates, in the second $n$ coordinates, and so on. Thus, the number of triangles in $X^k \times Y^k \times Z^k$ is $\delta^k N^{2k}=\epsilon^{(1+\alpha)C_pk}N^{2k}$. However, by Theorem \ref{mainasympt} applied to $X^k \times Y^k \times Z^k$, we must have 
\[ (\epsilon^k)^{C_p +o(1)} \le \delta^k = (\epsilon^k)^{(1+\alpha)C_p}, 
\]
where the $o(1)$ term tends to $0$ as $\epsilon^k$ tends to $0$. 
Letting $k \rightarrow \infty$, we obtain a contradiction. We conclude that $\delta \geq \epsilon^{C_p}$, completing the proof of Theorem \ref{main1}.
\end{proof}


\section{Proof of the upper bound}
For completeness, we now give a proof of the last part of Theorem \ref{main}, which shows that the lower bound is tight.

For $p=2$, using their lower bound construction discussed in Theorem \ref{sumsetresult}, Fu and Kleinberg \cite{FK14} provided a simple construction that implies an upper bound on $\delta$ in terms of $\epsilon$, which shows that the exponent $C_2$ in Theorem \ref{main} is tight. That a lower bound for the multicolor sum-free problem can be turned into an upper bound for the arithmetic removal lemma was first observed by Bhattacharyya and Xie  \cite{BX09}; see also \cite{HX15}. The argument works in general for $\mathbb{F}_p$, so by the result of Kleinberg, Sawin, and Speyer \cite{KSS16}, the exponent $C_p$ is tight for each $p$.

We now present a construction based on that of Bhattacharyya and Xie \cite{BX09}. Recall that the multicolored sum-free problem asks to find a collection of ordered triples $\{(x_i,y_i,z_i)\}_{i=1}^m$ in $\mathbb{F}_p^n$ such that $x_i+y_j+z_k=0$ holds if and only if $i=j=k$. We show that if there is such a collection of $m=p^{(1-c_p)n-o(n)}$ triples, then there are examples that show that in the removal lemma over $\mathbb{F}_p$, we must take 
\[\delta \le \epsilon^{C_p-o(1)}
.\]
Given such a set of $m$ triples $(x_i,y_i,z_i)$, with $X=\{x_i\}_{i=1}^m$, $Y=\{y_i\}_{i=1}^m$, $Z=\{z_i\}_{i=1}^m$, define, for any positive integer $l$, subsets $X',Y',Z' \subset \mathbb{F}_p^{n+l}$ by taking $X'=X \times \mathbb{F}_p^l$, $Y'=Y \times \mathbb{F}_p^l$, and $Z'=Z \times \mathbb{F}_p^l$. 

First, we claim that if we want to delete elements from $X'$, $Y'$, and $Z'$ such that no triangles remain, we must delete at least $mp^l$ elements in total. Fix $i \in [m]$, and look at $x_i \times \mathbb{F}_p^l, y_i \times \mathbb{F}_p^l, z_i \times \mathbb{F}_p^l$. These are all contained in $X'$, $Y'$, and $Z'$ respectively, so suppose that we have removed less than $p^l$ of these $3 \cdot p^l$ elements. Define $X_i$, $Y_i$, and $Z_i$ so that we have kept $x_i \times X_i$, $y_i \times Y_i$, and $z_i \times Z_i$, and note that they must each be nonempty. Since $x_i+y_i+z_i=0$, we must have that $X_i+Y_i$ is disjoint from $Z_i$. However, if we take any $z \in Z_i$, the set $z-Y_i$ must intersect the set $X_i$, since $|X_i|+|Y_i|>p^l$. This means that we must have at least one triangle, a contradiction.

Furthermore, we claim that the total number of triangles in $X' \times Y' \times Z'$ is $mp^{2l}$. Indeed, if $(x_i,x')+(y_j,y')+(z_k,z')=0$, then in particular, we must have $x_i+y_j+z_k=0$, so $i=j=k$, and in this case, we can choose anything in $\mathbb{F}_p^l$ for $x'$ and for $y'$, and then the choice of $z'$ is determined.

This means that taking 
\[\epsilon = \frac{mp^l}{p^{n+l}}=\frac{m}{p^n},\hspace{1cm}\delta = \frac{mp^{2l}}{p^{2n+2l}}=\frac{m}{p^{2n}},
\]
we obtain a construction of an infinite family of sets $X_l,Y_l,Z_l \subset \mathbb{F}_p^{n+l}$ such that there are at most $\delta p^{2n+2l}$ triangles between the sets, but we have to remove at least $\epsilon p^{n+l}$ points to remove all triangles. 

Since
\[m=p^{(1-c_p)n-o(n)},
\]
we have
\[\epsilon=p^{-n(c_p + o(1))},
\hspace{1cm}
\delta=p^{-n((1+c_p) + o(1))}.
\]
Thus, as $n \rightarrow \infty$, we have that $\epsilon \rightarrow 0$, and (recalling that $C_p=1+1/c_p$), we have
\[\delta \le \epsilon^{C_p-o(1)}.
\]




\section{Concluding remarks} 
We expect that these methods will also apply to similarly give tight polynomial bounds for removal lemmas of linear equations in more variables in vector spaces over $\mathbb{F}_p$. Sets that contain no arithmetic progressions have also been studied over other abelian groups, see \cite{BCCGNSU16} and \cite{PETROV16}. We expect these methods can be used in those settings as well to improve bounds on the corresponding removal lemma. We also expect further applications in property testing. 

The triangle removal lemma states that every graph on $n$ vertices with $o(n^3)$ triangles can be made triangle-free by removing $o(n^2)$ edges. One consequence of this, known as the diamond-free lemma, is that every graph on $n$ vertices in which each edge is in precisely one triangle has $o(n^2)$ edges. The diamond-free lemma is also equivalent to the Ruzsa-Szemer\'edi induced matching theorem, or the extremal $(6,3)$-theorem. The argument presented here for getting Green's arithmetic removal lemma from the tri-colored sum-free result can be adapted, by replacing random subspaces by random subsets, to show that the diamond-free lemma implies the triangle removal lemma if the bound on the diamond-free lemma is sufficiently good, and any such bound on the diamond-free lemma can be turned into a similar bound on the triangle removal lemma. 

\noindent {\bf Acknowledgements.} We would like to thank Lisa Sauermann for many helpful comments, including two simplifications of our proof. In particular, she pointed out that a lower bound on the degrees in Lemma \ref{roughlyeven} is not necessary. We would also like to thank Noga Alon, Thomas Church, Robert Kleinberg, Will Sawin, Benny Sudakov, Madhu Sudan, and Terence Tao for helpful comments.  
\bibliography{ref}

\bibliographystyle{amsplain_mod2}
\end{document}